\documentclass[12pt,twoside]{amsart}
\usepackage{amssymb}
\usepackage{graphicx}
\usepackage[margin=1in]{geometry}

\newcommand{\beas}{\begin{eqnarray*}}
\newcommand{\eeas}{\end{eqnarray*}}
\newcommand{\bea}{\begin{eqnarray}}
\newcommand{\eea}{\end{eqnarray}}
\newcommand{\beq}{\begin{equation}}
\newcommand{\eeq}{\end{equation}}
\newcommand{\ben}{\begin{enumerate}}
\newcommand{\een}{\end{enumerate}}

\newtheorem{theorem}{Theorem}
\newtheorem{lemma}[theorem]{Lemma}

\newtheorem{corollary}[theorem]{Corollary}
\newtheorem{conjecture}[theorem]{Conjecture}

\theoremstyle{definition}
\newtheorem{remark}[theorem]{Remark}

\usepackage{latexsym}
\usepackage{color,hyperref}
\definecolor{darkblue}{rgb}{0,0,0.6}
\definecolor{myred}{rgb}{1,0,0}
\hypersetup{colorlinks,breaklinks,linkcolor=darkblue,urlcolor=darkblue,anchorcolor=darkblue,citecolor=darkblue}

\author[James A. Sellers]{James A. Sellers}
\address{Department of Mathematics and Statistics, University of Minnesota Duluth, Duluth, MN 55812}
\email{jsellers@d.umn.edu}

\author[Fabrizio Zanello]{Fabrizio Zanello}
\address{Department of Mathematical Sciences, Michigan Tech, Houghton, MI 49931}
\email{zanello@mtu.edu}

\title[On the parity of the number of partitions with odd multiplicities]{On the parity of the number of partitions\\with odd multiplicities}

\begin{document}

\begin{abstract} 
{
Recently, Hirschhorn and the first author considered the parity of the function $a(n)$ which counts the number of integer partitions of $n$ wherein each part appears with odd multiplicity. They derived an effective characterization of the parity of $a(2m)$ based solely on properties of $m.$  In this note, we quickly reprove their result, and then extend it to an explicit characterization of the parity of $a(n)$ for all $n\not\equiv 7 \pmod{8}.$  We also exhibit some infinite families of congruences modulo 2 which follow from these characterizations.

We conclude by discussing the case $n\equiv 7 \pmod{8}$, where, interestingly, the behavior of $a(n)$ modulo 2 appears to be entirely different. In particular, we conjecture that, asymptotically, $a(8m+7)$ is odd precisely $50\%$ of the time. This conjecture, whose broad generalization to the context of eta-quotients will be the topic of a subsequent paper, remains wide open.
}
\end{abstract}


\maketitle

\section{Introduction} 

In this note, we contribute to the study of the parity of an interesting and naturally-defined partition function, denoted by $a(n)$, which counts the number of partitions of an integer $n$ whose parts all appear with odd multiplicity. We refer the reader to Andrews \cite{Andr} for basic facts and terminology of partition theory.

The function $a(n)$ has been considered by multiple authors since at least the 1950's \cite{ASA}, but a systematic investigation of its parity has only recently begun, thanks to a paper of Hirschhorn and the first author \cite{HS}. Their main theorem established the 
fact that $a(2m)$ is odd precisely when $m=0$ or $m$ is a square not divisible by 3. In the first portion of this note, we reprove and extend the result of \cite{HS}, by explicitly characterizing the integers $n$ for which $a(n)$ is odd, provided $n\not\equiv 7 \pmod{8}$. Interestingly, for such values of $n$, $a(n)$ turns out to be \emph{almost always} even (that is, odd with density zero). All of our techniques will be elementary, and involve a careful analysis of certain quadratic representations.

We then briefly discuss the behavior of $a(n)$ when $n\equiv 7 \pmod{8}$. Within this arithmetic progression, the situation appears to be dramatically different. In particular, we conjecture that $a(8m+7)$ is odd with density $1/2$. As a curious consequence, we expect $a(n)$ to be odd for $6.25\%$ of the values of $n$. Consistent with the state of the art on the parity of the coefficients of other eta-quotients, including the ordinary partition function $p(n)$, unfortunately our conjecture appears extremely hard to attack with existing methods.

\section{Results on the parity of $a(n)$}

Let $a(n)$ denote the number of partitions of $n$ where each part appears with odd multiplicity.  For instance, $a(5)=5$, since the partitions $(5), (4,1), (3,2), (2,1,1,1)$, and $(1,1,1,1,1)$ only present odd multiplicities for each of the parts, while the remaining two partitions of 5, $(3,1,1)$ and $(2,2,1)$, each contain a part with even multiplicity. It is clear that the generating function for $a(n)$ is given by
\begin{equation}\label{gf}
\sum_{n\ge 0}a(n)q^n = \prod_{i\ge 1}\left(1+\frac{q^i}{1-q^{2i}}\right) =\prod_{i\ge 1}\frac{1+q^i-q^{2i}}{1-q^{2i}}.
\end{equation} 

We let
$$f_r(q)=f_r=\prod_{i\ge 1}(1-q^{ri}),$$
and say that two series $F(q)$ and $G(q)$ are \emph{congruent modulo 2}, writing  $F(q)\equiv G(q)$, if the coefficients of $q^n$ in the corresponding power series representations have the same parity, for all $n$. (Unless otherwise specified, all congruences in this paper will have modulo 2.) 

In the following lemma, we recall a few important facts. 

\begin{lemma}\label{2}
\begin{equation}\label{11}
f_1^3 \equiv f_3+ q f_9^3.
\end{equation} 

\begin{equation}\label{22}
f_1^3f_3^3\equiv f_1^{12} + q f_3^{12}.
\end{equation}

\begin{equation}\label{33}
\frac{f_3^3}{f_1} \equiv \sum_{n\in \mathbb Z} q^{n(3n-2)}.
\end{equation}
\end{lemma}

\begin{proof} Formula (\ref{11}) easily follows from Euler's Pentagonal Number Theorem, namely 
$$f_1= \sum_{n\in \mathbb Z} (-1)^nq^{n(3n-1)/2}\equiv \sum_{n\in \mathbb Z} q^{n(3n-1)/2},$$ 
along with Jacobi's formula 
$$f_1^3\equiv  \sum_{n\ge 0} q^{n(n+1)/2},$$ 
since every integer is congruent to precisely one of 0, 1, or 2 modulo 3. (See also \cite[Chapter 1]{Hi}.)
Formula (\ref{22}) is essentially \cite[Theorem 2.3]{BBG}, while (\ref{33}) is \cite[Corollary 1]{Ro}.
\end{proof}

We now significantly extend the main result of \cite{HS} (which corresponds to Theorem \ref{characterization2} below), by characterizing the values of $n$, outside of the arithmetic progression $n\equiv 7 \pmod{8}$, such that $a(n)$ is odd.  We begin by providing a shorter proof of the result of \cite{HS}, namely a characterization of the parity of $a(2m)$.  

\begin{theorem}\label{characterization2} 
For all $m\geq 0,$ $a(2m)$ is odd if and only if  $m = 0$ or $m =k^2$ for some integer $k$ not divisible by 3.
\end{theorem}

\begin{proof}
Since we are working modulo 2, the generating function (\ref{gf}) becomes
$$\sum_{n\ge 0}a(n)q^n =\prod_{i\ge 1}\frac{1+q^i-q^{2i}}{1-q^{2i}}\equiv \prod_{i\ge 1}\frac{1+q^i+q^{2i}}{1-q^{2i}}\equiv \frac{f_3}{f_1^3}.$$
Dividing across (\ref{22}) by $f_1^6f_3^2,$ we obtain
\begin{equation}\label{351}
\frac{f_3}{f_1^3}\equiv  \frac{f_1^6}{f_3^2}+q \frac{f_3^{10}}{f_1^6}.
\end{equation}
Thus,  modulo 2, $a(2m)$ coincides with the coefficient of $q^m$ in the power series representation of $f_1^3/f_3$. By (\ref{11}) and (\ref{33}),
$$ \frac{f_1^3}{f_3}\equiv 1+q\frac{f_9^3}{f_3}\equiv 1+\sum_{n\in \mathbb Z} q^{1+3n(3n-2)}=1+\sum_{n\in \mathbb Z} q^{(3n-1)^2},$$
and the result follows.
\end{proof}

Next, we extend this result by obtaining an explicit characterization of the parity of $a(n)$ for the remaining values of $n\not\equiv 7 \pmod{8}$.  We begin by considering the behavior of $a(4m+1)$ modulo 2.  

\begin{theorem}
\label{characterization41}
Let the prime factorization of $4m+1$ be written as
$$4m+1 = \prod_{i\ge 1}p_i^{\alpha_i}.$$
Then, for all $m\geq 0,$  $a(4m+1)$ is odd if and only if

1) $m\equiv 0\pmod{3}$, and $4m+1$ is a square; or

2) $m\equiv 1\pmod{3}$, and all $\alpha_i$ are even except for exactly one of them, which is congruent to $1\pmod{4}$. 
\end{theorem}

\begin{proof}
We begin by noting that, by (\ref{351}),  $a(2m+1)$ is the same, modulo 2, as the coefficient of $q^m$ in the power series representation of $f_3^5/f_1^3$. Multiplying  across  (\ref{22}) by $f_3^2/f_1^6$ gives
\begin{equation}\label{73}
\frac{f_3^5}{f_1^3}\equiv f_1^6f_3^2+q\frac{f_3^{14}}{f_1^6}.
\end{equation}

Hence, by (\ref{73}), $a(4m+1)$ is congruent, modulo 2, to the coefficient of $q^m$ in the power series representation of $f_1^3f_3$. Therefore, from Euler's Pentagonal Number Theorem and Jacobi's formula (see the proof of Lemma \ref{2}), we obtain
$$f_1^3f_3\equiv  \sum_{a\ge 0} q^{a(a+1)/2} \cdot \sum_{b\in \mathbb Z} q^{3b(3b-1)/2}.$$
This means that, modulo 2, $a(4m+1)$ is the number of representations of $m$ of the form
$$
m = \frac{a(a+1)}{2} + \frac{3b(3b-1)}{2}
$$
with $a\geq 0$ and $b\in \mathbb Z.$  
By completing the square, this is equivalent to considering representations of $8m+2$ of the form 
$$
8m+2 = (4a^2+4a+1) + (36b^2 - 12b + 1) = (2a+1)^2 + (6b-1)^2,
$$
again with $a\geq 0$ and $b\in \mathbb Z.$  

Note that in representing $8m+2$ as
$$8m+2=c^2 + d^2$$
with $d$ not divisible by 3, we have $8m+2 \not\equiv 0\pmod{3}$, since $c^2 \equiv 0 \mbox{\ or\ } 1 \pmod{3}$  while $d^2$ can only be $1 \pmod{3}.$ Hence, $m\not\equiv 2\pmod{3}$.

Next, assume $8m+2 \equiv 2 \pmod{3}.$  Then, $m\equiv 0 \pmod{3}.$  This means that $c^2 \equiv d^2 \equiv 1 \pmod{3}$.  By symmetry of the representations of $8m+2$ (where $c$ and $d$ are interchangeable), it is clear that the desired number of ways to represent $8m+2$ as $c^2+d^2$, which we denote by $R_2(8m+2),$ is odd if and only if $c=d$. Hence
$$8m+2 = 2(4m+1)= 2c^2.$$
It follows that, for $m\equiv 0\pmod{3},$ $a(4m+1)$ is odd precisely when $4m+1$ is a square, as we wanted to show.

Finally, assume $8m+2 \equiv 1 \pmod{3},$ so that $m\equiv 1 \pmod{3}.$  Then $c \equiv 0 \pmod{3}$ since $d\not \equiv 0 \pmod{3}$.  Thus, since we no longer have symmetry in $c$ and $d$ in our representations, and $8m+2$ cannot be a square, it is easy to see that
$$
R_2(8m+2) = \frac{r_2(8m+2)}{8},
$$
where $r_2(8m+2)$ counts the number of arbitrary integer representations of
$$8m+2=2(4m+1)=2 \prod_{i\ge 1}p_i^{\alpha_i}$$
as a sum of two squares.

By known facts about $r_2(n)$ (see for instance \cite[Chapter 2]{Hi}) and some straightforward algebra, we conclude in this case that $a(4m+1)$ is odd if and only if all $\alpha_i$ are even except for precisely one of them, which is congruent to $1\pmod{4}$. This completes the proof.  
\end{proof} 

\begin{remark}\label{41r}
We now present an alternative proof to Theorem \ref{characterization41} for the case $m\equiv 0 \pmod{3}$. We want to show that $a(4m+1)$ is odd if and only if $4m+1$ is a square. We use the fact, shown at the beginning of the proof of the theorem, that $a(4m+1)$ coincides, modulo 2, with the coefficient of degree $m$ in the power series representation of $f_1^3f_3$.

Thus,  by (\ref{11}),
$$f_1^3 f_3 \equiv (f_3 + q f_9^3) f_3 \equiv f_3^2 + q f_3 f_9^3 \pmod{2}.$$

The contribution to the right side coming from the degrees $m\equiv 0 \pmod{3}$ is therefore given by the term
$$f_3^2 \equiv \sum_{n\in \mathbb Z} q^{3n(3n-1)} \pmod{2}.$$
Hence, when $m\equiv 0 \pmod{3}$, $a(4m+1)$ is odd if and only if
$$4m+1=4(3n(3n-1))+1=(6n-1)^2,$$
for $n\in \mathbb Z$. Since $4m+1$ is coprime to 6 for $m\equiv 0 \pmod{3}$, the last displayed formula means precisely that $4m+1$ is a square, as we wanted to show.
\\
\end{remark}

Theorem \ref{characterization41} can be used to easily prove infinitely many Ramanujan-like congruences modulo 2 satisfied by the function $a$ in arithmetic subprogressions of $4m+1.$  We outline some of these congruences here.  

\begin{corollary}
\label{example1}
For all $n\geq 0,$ $$a(12n+9) \equiv 0 \pmod{2}.$$
\end{corollary}

\begin{proof}
Note that $12n+9 = 4(3n+2)+1$, so $m=3n+2 \equiv 2 \pmod{3}.$  Given Theorem \ref{characterization41}, we see that this result holds.  
\end{proof}

\begin{corollary}
\label{example2}
For all $n\geq 0,$ $$a(24n+13) \equiv 0 \pmod{2}.$$
\end{corollary}

\begin{proof}
Note that $24n+13 = 4(6n+3)+1$, so $m=6n+3 \equiv 0 \pmod{3}.$  Thus, thanks to Theorem \ref{characterization41}, we must ask whether $24n+13$ can ever be a square.  However, $24n+13 \equiv 5 \pmod{8},$ while all squares must be congruent to 0, 1, or $4\pmod{8},$ and the result follows.  
\end{proof}

\begin{corollary}
\label{example3}
Let $p\geq 5$ be prime, and let $r,$ $1\leq r\leq p-1,$ be such that $12r+1$ is a quadratic nonresidue modulo $p.$  Then, for all $n\geq 0,$ $$a(12pn+12r+1) \equiv 0\pmod{2}.$$  
\end{corollary}

\begin{proof}
Note that $m=3(pn+r),$ so $m\equiv 0\pmod{3}.$  We must check whether $12pn+12r+1 = 4(3(pn+r))+1$ is a square.  Notice that 
$$
12pn+12r+1 \equiv 12r+1 \pmod{p}
$$ 
and $12r+1$ is assumed to be a quadratic nonresidue modulo $p.$  Therefore, $12pn+12r+1$ cannot be a square.  By Theorem \ref{characterization41}, the result follows.  
\end{proof}
So, for example, the following Ramanujan-like congruences hold for all $n\geq 0$: 
\allowdisplaybreaks
{
\begin{eqnarray*}
a(60n+13) &\equiv & 0 \pmod{2} \ \ \ (p=5, r=1) \\
a(60n+37) &\equiv & 0 \pmod{2} \ \ \ (p=5, r=3) \\
a(84n+13) &\equiv & 0 \pmod{2} \ \ \ (p=7, r=1) \\
a(84n+61) &\equiv & 0 \pmod{2} \ \ \ (p=7, r=5) \\
a(84n+73) &\equiv & 0 \pmod{2} \ \ \ (p=7, r=6) \\
a(132n+13) &\equiv & 0 \pmod{2} \ \ \ (p=11, r=1) \\
a(132n+61) &\equiv & 0 \pmod{2} \ \ \ (p=11, r=5) \\
a(132n+73) &\equiv & 0 \pmod{2} \ \ \ (p=11, r=6) \\
a(132n+85) &\equiv & 0 \pmod{2} \ \ \ (p=11, r=7) \\
a(132n+109) &\equiv & 0 \pmod{2} \ \ \ (p=11, r=9) 
\end{eqnarray*}
{\ }\\
We now turn our attention to a characterization of the parity of $a(8m+3).$

\begin{theorem}
\label{characterization83}
Let the prime factorization of $8m+3$ be written as 
$$8m+3 =  \prod_{i\ge 1} p_i^{\beta_i}.$$
Then, for all $m\geq 0,$ 
$a(8m+3)$ is odd if and only if

1) $m\equiv 0\pmod{3},$ and $8m+3$ is 3 times a square; or

2) $m\equiv 1\pmod{3},$ and all $\beta_i$ are even except for exactly one of them, which is congruent to $1\pmod{4}$.

\end{theorem}

\begin{proof}
Note that, by (\ref{73}), the parity of $a(4m+3)$ is the same as that of the coefficient of $q^m$ in the power series representation of $f_3^7/f_1^3$. Multiplying across  (\ref{22}) by $f_3^4/f_1^6$ yields
\begin{equation}\label{83}
\frac{f_3^7}{f_1^3}\equiv f_1^6f_3^4+q\frac{f_3^{16}}{f_1^6}.
\end{equation}

Therefore, $a(8m+3)$ is given, modulo 2, by the coefficient of $q^m$ in the power series representation of $f_1^3f_3^2$. Arguing similarly to the previous case, we obtain
$$f_1^3f_3^2\equiv  \sum_{a\ge 0} q^{a(a+1)/2}  \cdot \sum_{b\in \mathbb Z} q^{3b(3b-1)}.$$

By completing the square and using some simple algebra, our problem becomes equivalent to determining when the number of representations of $8m+3$ as
$$8m+3=c^2 + 2d^2$$
is odd, for $c$ and $d$ positive integers and $d\not\equiv 0\pmod{3}.$

First note that $m \not\equiv 2 \pmod{3}.$   Otherwise, $8m+3 \equiv 1 \pmod{3},$ but since squares can only be congruent to 0 or $1\pmod{3}$, this would force $d\equiv 0 \pmod{3}$ against our assumption.

Now assume that $m \equiv 1\pmod{3}.$   It follows that $8m+3 \equiv 2\pmod{3}$ and that $d^2$ is necessarily $1\pmod{3}$. Thus, in this case, our number of representations coincides with the number, say $S(8m+3)$, of arbitrary positive representations of $8m+3$ as $c^2 + 2d^2$.

Since $8m+3$ can neither be a square nor twice a square, we deduce that $S(8m+3)$ is exactly $1/4$ of the total number of integer representations of $8m+3$ as $c^2 + 2d^2$.
It is a classical fact going back to Dirichlet (see e.g. \cite{Hi1}) that the latter number is  
$$2 (d_{1,8}(8m+3) + d_{3,8}(8m+3) - d_{5,8}(8m+3) - d_{7,8}(8m+3)),$$
where $d_{i,8}(8m+3)$ counts the number of divisors of $8m+3$ that are congruent to $i \pmod{8}$.

Now note that the divisors of $8m+3$ congruent to 5 and those congruent to $7 \pmod{8}$ are in obvious bijection, since $5 \times 7 \equiv 3\pmod{8}.$
Therefore, 
\begin{align*}
S(8m+3) 
&=  2 (d_{1,8}(8m+3) + d_{3,8}(8m+3) - d_{5,8}(8m+3) - d_{7,8}(8m+3)) / 4 \\
&\equiv   2 (d_{1,8}(8m+3) + d_{3,8}(8m+3) + d_{5,8}(8m+3) + d_{7,8}(8m+3)) / 4 \pmod{2},
\end{align*}
where the latter formula is obviously half the number of all the positive divisors of $8m+3$. As usual, we denote this number by
$\sigma_0 (8m+3)$.

Recall that the prime factorization of $8m+3$ is 
$$8m+3 =  \prod_{i\ge 1} p_i^{\beta_i}.$$ 
It is a known fact that 
$$\sigma_0 (8m+3) = \prod_{i\ge 1} (\beta_i + 1).$$
Hence, we can easily see that
$$S(8m+3)\equiv \sigma_0 (8m+3) / 2 \equiv 1 \pmod{2}$$
if and only if all of the $\beta_i$ are even, except for precisely one which must be $1 \pmod{4}$. This concludes the proof when $m \equiv 1 \pmod{3}.$ 

Finally, let $m\equiv 0 \pmod{3}.$ Then $8m+3\equiv 0 \pmod{3}$ as well. Here we will provide an argument along the same lines as (though a little more complicated than) the previous case; we will sketch a different but shorter second proof in Remark \ref{ss}.
 
First note that, if $m$ is 0 or $6\pmod{9}$, $8m+3$ is divisible by 3 but not by 9, and therefore, in representing
$$8m+3=c^2+2d^2,$$
we necessarily have $d\not\equiv 0\pmod{3}.$  Otherwise, $c^2\equiv 0 \pmod{9}$, which implies $8m+3 \equiv 0 \pmod{9},$ a contradiction.

Hence, if $m$ is 0 or $6 \pmod{9}$, the desired representations of
$$8m+3=c^2+2d^2$$
coincide with the arbitrary positive representations of the same form. An argument entirely similar to the above now yields that the parity of such representations is the same as that of
$$\sigma_0(8m+3) / 2 =\left( \prod_{i\ge 1} (\beta_i +1) \right) / 2.$$
Notice in this case that $\beta_1 = 1$, since it is the exponent of 3 in the factorization of $8m+3$. 

Thus, when $m$ is 0 or $6 \pmod{9}$, we conclude that the desired number of representations is odd if and only if, for all $i\ge 2$, the $\beta_i$ are even. This is equivalent to saying that $8m+3$ is 3 times a square, which is the theorem.

The last case to consider is when $m\equiv 3\pmod{9}.$  Set $m=9w+3$, for some integer $w.$ We see that
$$8m+3 = 8(9w+3) + 3 = 9(8w+3).$$
Note that when $d\equiv 0 \pmod{3},$ $c\equiv 0\pmod{3}$ as well. Thus,
$$8m+3 = c^2 + 2d^2$$
becomes
$$8w+3 = c_1^2 + 2d_1^2,$$
where $c_1 = c/3$, and $d_1 = d/3.$

We deduce that the desired number of representations of $8m+3$ as $c^2 + 2d^2$, with $d$ not divisible by 3, coincides with the number of arbitrary positive representations of $8m+3$ minus the number of arbitrary positive representations of $8w+3$ of the same form. 

Now we notice that
$$8m+3 = 3^{\beta_1} \cdot \prod_{i\ge 2} p_i^{\beta_i},$$
where $\beta_1 \ge 2$ since $8m+3 \equiv 0 \pmod{9}.$ Hence
$$8w+3 = (8m+3) / 9 =  3^{\beta_1 - 2} \cdot \prod_{i\ge 2} p_i^{\beta_i}.$$

Arguing as above, it suffices to evaluate the parity of 
$( \sigma_0 (8m+3) - \sigma_0 (8w+3) ) /2.$
But the latter expression equals
$$( (\beta_1+1)  \cdot \prod_{i\ge 2} (\beta_i+1) -  (\beta_1 - 1)  \cdot \prod_{i\ge 2} (\beta_i+1) ) /2 =  \prod_{i\ge 2} (\beta_i+1) .$$

Thus, the desired number is odd if and only if the last product is odd, which happens precisely when all $\beta_i$ are even for $i\ge 2.$
Finally, since $8m+3$ cannot be a square, the exponent $\beta_1$ of 3 must be odd. Thus, $8m+3$ is 3 times a square, as we wanted to show.
\end{proof}

\begin{remark}\label{ss}
We now outline an alternative proof of Theorem \ref{characterization83} specifically for the case $m\equiv 0 \pmod{3}$, similarly to what we did in Remark \ref{41r} for the corresponding case of Theorem \ref{characterization41}. This proof is of a different nature but more straightforward than the one we included in the theorem, hence we believe also interesting.

We want to show that $a(8m+3)$ is odd if and only if $8m+3$ is 3 times a square. We employ a fact stated at the beginning of the  proof of Theorem \ref{characterization83}, namely that $a(8m+3)$ has the same parity of the coefficient of degree $m$ in the power series representation of $f_1^3 f_3^2$. Thus, arguing similarly to Remark \ref{41r}, note that, by (\ref{11}),
$$f_1^3 f_3^2 \equiv (f_3 + q f_9^3) f_3^2 \equiv f_3^3 + q f_3^2 f_9^3 \pmod{2}.$$

Hence, the contribution to the right side coming from the degrees $m\equiv 0 \pmod{3}$ is given by the term
$$f_3^3 \equiv \sum_{n\ge 0} q^{3n(n+1) / 2} \pmod{2}.$$
It follows that, when $m\equiv 0 \pmod{3}$, $a(8m+3)$ is odd if and only if
$$8m+3 = 8 (3n(n+1) / 2) +3 = 3 (2n+1)^2,$$
for $n\ge 0$. But this occurs precisely when $8m+3$ is 3 times a square, as desired.
\\
\end{remark}

We close this section by highlighting two corollaries which provide specific Ramanujan-like congruences for the function in question within the arithmetic progression $8m+3.$  

\begin{corollary}
For all $n\geq 0,$ $$a(24n+19)\equiv 0\pmod{2}.$$  
\end{corollary}

\begin{proof}
Note that $24n+19 = 8(3n+2)+3,$ so $m= 3n+2$ in this case.  Since $m\equiv 2\pmod{3},$ the result follows immediately from Theorem \ref{characterization83}.
\end{proof}

\begin{corollary}
\label{corollary24-3}
Let $p\geq 3$ be prime and let $r$ be an integer, $1\leq r\leq p-1,$ such that $8r+1$ is a quadratic nonresidue modulo $p.$  Then, for all $n\geq 0,$ $$a(24pn+24r+3) \equiv 0 \pmod{2}.$$  
\end{corollary}

\begin{proof}
Note that $24pn+24r+3 = 8(3pn+3r)+3,$ so $m=3(pn+r)$ in this case.  Here, $m\equiv 0\pmod{3},$ so by Theorem \ref{characterization83}, we need to determine whether $24pn+24r+3$ is three times a square.  Equivalently, we need to know whether $8pn+8r+1$ is a square.  However, by our assumption, 
$$
8pn+8r+1 \equiv 8r+1 \pmod{p}
$$ 
cannot be a square because $8r+1$ is a quadratic nonresidue modulo $p.$  The result follows.  
\end{proof}

Thus, for example, thanks to Corollary \ref{corollary24-3}, the following Ramanujan-like congruences hold for all $n\geq 0$: 
\allowdisplaybreaks
{
\begin{eqnarray*}
a(72n+51) &\equiv & 0 \pmod{2} \ \ \ (p=3, r=2) \\
a(120n+51) &\equiv & 0 \pmod{2} \ \ \ (p=5, r=2) \\
a(120n+99) &\equiv & 0 \pmod{2} \ \ \ (p=5, r=4) \\
a(168n+51) &\equiv & 0 \pmod{2} \ \ \ (p=7, r=2) \\
a(168n+99) &\equiv & 0 \pmod{2} \ \ \ (p=7, r=4) \\
a(168n+123) &\equiv & 0 \pmod{2} \ \ \ (p=7, r=5) \\
a(264n+51) &\equiv & 0 \pmod{2} \ \ \ (p=11, r=2) \\
a(264n+123) &\equiv & 0 \pmod{2} \ \ \ (p=11, r=5) \\
a(264n+171) &\equiv & 0 \pmod{2} \ \ \ (p=11, r=7) \\
a(264n+195) &\equiv & 0 \pmod{2} \ \ \ (p=11, r=8) \\
a(264n+219) &\equiv & 0 \pmod{2} \ \ \ (p=11, r=9) 
\end{eqnarray*}

\section{A conjecture on the parity of $a(8m+7)$}

We wrap up our paper by presenting a conjecture on the parity of $a(n)$ within the arithmetic progression $n=8m+7$, which is the only case that was not considered in the characterizations of the previous section. 

First recall that the \emph{density} of the odd values of a sequence $s_n$ is defined by
$$\lim_{x \rightarrow \infty} \frac{\# \{n \leq x : s_n {\ }\text{is odd} \}}{x},$$
if such limit exists. The following is an important consequence of Theorems \ref{characterization2}, \ref{characterization41}, and \ref{characterization83}.

\begin{corollary}\label{cor}
Outside of the arithmetic progression $n\equiv 7\pmod{8}$, $a(n)$ is odd with density zero. In particular, the entire sequence $a(n)$ is odd with density at most $1/8$, if such density exists.
\end{corollary}

\begin{proof}
When $n$ is even the result is clear by Theorem \ref{characterization2}, since squares have density zero in the integers.

As for $n\equiv 1 \pmod{4}$ or $n\equiv 3 \pmod{8}$, we have from the proofs of Theorems \ref{characterization41} and \ref{characterization83} that $a(n)$ is odd only if $n$ has suitable binary quadratic representations. But a classical theorem of Landau \cite{Land,SerLan} gives that only the order of $x/\sqrt{\log x}$ integers $n\le x$ can have such representations. Thus, the desired density-zero conclusion follows \emph{a fortiori} once we reduce modulo 2. This proves the corollary.
\end{proof}

Let us now turn to the arithmetic progression $n=8m+7$. Interestingly, here the behavior of our function modulo 2 appears to be extremely different. In fact, extensive computational data (for which we thank William Keith) leads us to conjecture the following.

\begin{conjecture}\label{conj}
The coefficients $a(8m+7)$ are odd with density $1/2$. Equivalently, $a(n)$ is odd with density $1/16$.
\end{conjecture}

\begin{remark}
Using (\ref{83}), Conjecture \ref{conj} can be restated by saying that the coefficients in the power series representation of
\begin{equation}\label{3813}
\frac{f_3^8}{f_1^3}
\end{equation}
are odd with density $1/2$. Unfortunately, proving this fact still appears entirely out of reach today. Indeed -- in analogy with our current level of understanding of the coefficients in the power series representation of $1/f_1$, namely the ordinary partition function $p(n)$, for which the same odd density of $1/2$ was conjectured \cite{Calk,PaSh} -- we have not even been able to show that the odd density of $a(8m+7)$ must exist, or that it is positive.

We only remark here that a study of the parity of the coefficients of \emph{eta-quotients} (of which both $1/f_1$ and $f_3^8/f_1^3$ are interesting but specific instances), as well as some general conjectures on the behavior of those coefficients, will be the focus of an upcoming paper by the second author and W. Keith \cite{KZ}.\\
\end{remark}

\section*{Acknowledgements} The idea of this work originated during a visit of the first author to Michigan Tech in Spring 2020. The second author thanks his former PhD student Samuel Judge and William Keith for several discussions on the broader topic of this paper. The second author was partially supported by a Simons Foundation grant (\#630401).


\end{document}